\documentclass[reqno,11pt]{amsart}
\usepackage[utf8]{inputenc}
\usepackage{amsmath,amssymb,amsthm}
\usepackage{tikz}
\usetikzlibrary{arrows}
\usepackage{url}
\usepackage[utf8]{inputenc}
\usepackage{amsrefs}
\newtheorem{theorem}{Theorem}[section]

\newtheorem{remark}[theorem]{Remark}
\newtheorem{lemma}[theorem]{Lemma}

\newtheorem{question}[theorem]{Question}

\title{Primitive and Geometric-Progression-Free Sets without large gaps}
\author{Nathan McNew}
\address{Department of Mathematics, Towson University, 8000 York Road, Towson, MD 21252}
\email{nmcnew@towson.edu}

\AtBeginDocument{%
   \def\MR#1{}
}

\begin{document}
\begin{abstract}
     We prove the existence of primitive sets (sets of integers in which no element divides another) in which the gap between any two consecutive terms is substantially smaller than the best known upper bound for the gaps in the sequence of prime numbers.  The proof uses the probabilistic method.  Using the same techniques we improve the bounds obtained by He for gaps in geometric-progression-free sets.  
\end{abstract}
\maketitle

\section{Introduction}

Despite the rich history of research on the gaps in the sequence of prime numbers, including many recent breakthroughs, the magnitudes of the largest gaps in this sequence are still poorly understood.  Denoting by $p_1, p_2, \ldots$ the sequence of prime numbers, it has been known since 2001, due to Baker, Harman, and Pintz \cite{bhp}, that \[p_n -p_{n-1} \ll p_n^{0.525}.\]
Assuming the Riemann Hypothesis gives a small improvement. Cram\'er \cite{cra} shows 
 \[p_n -p_{n-1} \ll \sqrt{p_n}\log p_n.\]
Cram\'er \cite{Cra2} conjectures, however, that the bound $p_n -p_{n-1} \ll \log^2 p_n$ gives the true order of magnitude of the largest gaps. As for lower bounds, it follows immediately from the prime number theorem that there must exist gaps where $p_n -p_{n-1} \geq \log p_n$.  This can be improved upon slightly.  It has recently been shown by Ford, Green, Konyagin, Maynard and Tao \cite{fgkmt} that, for some positive constant $c$, the innequality \[p_n-p_{n-1} > \frac{c\log p_n \log \log  p_n \log_4p_n}{\log_3 p_n} \] holds infinitely often, improving on the previous result of Rankin \cite{rankinprimes} which included an additional triple $\log $ factor in the denominator.  Here, and throughout the paper, $\log_i x$ will be used to denote the $i$-fold iterated logarithm when $i\geq 3$. Since $\log \log x$ is commonly used it will be used for readability when $i=2$. 

Generalizing from the set of primes, one can consider any primitive set of integers.  We say a set is primitive if no integer in the set divides another integer in the set.  The study of primitive sets also has a rich history. For example, it is known that primitive sets can have counting function substantially larger than the prime numbers.  Ahlswede, Khachatrian, and S\'ark\"{o}zy \cite{AKS} showed there exists a primitive sequence $s_1<s_2<\cdots$ with \[n \asymp \frac{s_n}{\log \log  s_n (\log_3 s_n)^{1+\epsilon}}\] for sufficiently large $n$.   Martin and Pomerance \cite{mp} show that this can be improved slightly, in fact there exists such a sequence where 
\[n \asymp \frac{s_n}{\log \log  s_n \log_3 s_n \cdots \log_k s_n (\log_{k+1} s_n)^{1+\epsilon}}\] 
for sufficiently large $n$ and any $k\geq 2$.  This is, in a sense, best possible, as Erd\H{o}s \cite{erdosprimitive} shows that any primitive sequence $s_1, s_2, \ldots$ must satisfy \[\sum_{n=1}^\infty \frac{1}{s_n \log s_n} < \infty.\]
Compared to the sequence of prime numbers, where the average gap grows like $\log x$, we see from these results that primitive sets can have substantially smaller gaps on average, on the order of $\log \log  x \log_3 x \cdots \log_k x (\log_{k+1} x)^{1+\epsilon}$ for any $k\geq 2$.  Nevertheless, it has not yet been possible to show that the largest gaps among these sequences is any smaller than what is known for the prime numbers.    

We show here that there exist primitive sequences in which the gap between consecutive terms is substantially smaller than has been previously shown for the primes or any other primitive sequence.  In particular, we get the following upper bound.

\begin{theorem} \label{thm:primitive}
For any $\epsilon>0$ there exists a primitive sequence $q_1< q_2 < \cdots$ of integers in which the gap between any two consecutive terms is bounded above by \begin{equation} q_n-q_{n-1} \leq \exp \left(\sqrt{2\log q_n \log \log  q_n + (2+\epsilon)\log q_n \log_3 q_n}\right). \label{primitive bound} \end{equation}
\end{theorem}

The proof utilizes the probabilistic method, and so it is not constructive.  It generalizes, however, to the related problem of geometric-progression-free sets, where the analogous problem has recently attracted attention.  

If $r>1$ is rational (sometimes we insist it be integral), then a geometric progression of length $k$ with ratio $r$ is a progression of integers $(g_1,g_2,\ldots g_k$) in which $g_i=rg_{i-1}$.  We say $S$ avoids geometric progressions of length $k$ if it is not possible to find $k$ integers from $S$ in a geometric progression.  Note that primitive sets can be described as sets avoiding geometric progressions of length 2 in which we insist that the ratio $r$ must be an integer.  For the remainder of the paper we will assume that our geometric progressions have length at least 3, and, unless otherwise stated, are allowed to have rational ratio. 

In the case of geometric-progression-free sets, unlike primitive sets, there exist such sets with positive density.  In particular, the squarefree numbers avoid geometric progressions and have density $\frac{6}{\pi^2}$, though this density isn't best possible.  (See \cites{mcnewgpf,NO,Rankin} for results on the maximum density of such a set.)  

Because of this it is not clear, a priori, that there cannot exist such sets in which all of the gaps are bounded above by a fixed constant.  In ergodic theory a set in which every gap is bounded by a constant is known as a \textit{syndetic} set.  Bieglb\"{o}ck, Bergelsen, Hindman and Strauss \cite{BBHS} first posed the question of whether there exists a syndetic set that is geometric-progression-free.  This problem has become well-known as a good example of the difficulty inherent in studying problems that mix the additive and multiplicative structure of the integers, and remains open.  

There has been partial progress toward this question for 2-syndetic sets (sets in which the difference between any two consecutive terms is at most two).  He \cite{He} shows by a computer search that any subset of the range [1,640] containing at least one of any pair of consecutive numbers must contain three term geometric progressions.  Recently Patil \cite{Patil} shows that any sequence of integers $s_1<s_2<\cdots$ with $s_n-s_{n-1} \leq 2$ must contain infinitely pairs $\{a,ar^2\}$ with $r$ an integer. 

In general, one can avoid geometric progressions of length $k{+}1$ by taking the sequence of $k$-free numbers.  Denoting by $s_1<s_2<\cdots$ the sequence of $k$-free numbers, the best known bound on the gaps, due to Trifonov \cite{trifonov} is that \[s_n-s_{n-1} \ll s_n^{\frac{1}{2k+1}} \log s_n.\]  Though this, again, is likely far greater than the truth.  

He \cite{He} considers the existence of geometric-progression-free sets with gaps provably smaller than the bounds for $k$-free numbers.  He shows the following.

\begin{theorem}[He]
For each $\epsilon>0$ there exists a sequence $b_1<b_2<\cdots$ avoiding 6-term geometric progressions satisfying \[b_n-b_{n-1} \ll_\epsilon \exp\left( \left(\frac{5\log 2}6 +\epsilon \right)  \frac{\log b_n}{\log \log  b_n}\right).\]
Furthermore, there exists a sequence $c_1<c_2<\cdots$ avoiding 5-term geometric progressions satisfying \[c_n-c_{n-1} \ll_\epsilon c_n^\epsilon\] 
and a sequence $d_1<d_2<\cdots$ that avoids 3-term geometric progressions with integral ratio in which \[d_n-d_{n-1} \ll_\epsilon d_n^\epsilon.\] 
\end{theorem}

The technique developed here allows us to treat 3-term geometric progressions with rational ratio and obtain a substantially smaller bound on the size of the gaps.  In particular we prove the following in section \ref{sec:gpf} .  

\begin{theorem} \label{thm:gpf}
For any $\epsilon>0$ there exists a sequence of integers $t_1< t_2< \cdots$ free of 3-term-geometric-progressions, such that \begin{equation} t_n-t_{n-1} \leq \exp \left(2\sqrt{\log 2\log t_n + \tfrac{3+\epsilon}{2}\sqrt{\log 2\log t_n}\log \log  t_n}\right). \label{gp bound} \end{equation}
\end{theorem}

\section{Coprime subsets of intervals}
We first prove that in any short interval we can find a relatively large subset of integers that are pairwise coprime. Using the linear sieve of Rosser and Iwaniec (see for example Theorem 12.14 and Corollary 12.15 of \cite{cribro}) one can sieve an interval of length $y$ by primes up to  nearly $\sqrt{y}$.  The result can be stated as follows.
\begin{lemma} \label{lem:sieve}
There exist positive constants $c_1$ and $c_2$ so that every interval of length $c_1y$ with $y\geq 2$ contains at least $\frac{y}{ \log^2 y}$ integers free of prime factors smaller than $\sqrt{y}$, and at most $\frac{c_2 y}{\log y}$ such integers.
\end{lemma}

Using this we can show that the short interval $[x-y,x]$ contains a reasonably large subset of pairwise coprime integers.  Erd\H{o}s and Selfridge \cite{erdself} (see also \cite{erdric}) prove that for sufficiently large $y$ and any $\epsilon>0$  any such interval has a pairwise coprime subset of size at least $y^{1/2-\epsilon}$, though their proof is not correct as written.  We correct and refine the argument, using Lemma \ref{lem:sieve} to show the following.   

\begin{theorem} \label{thm:coprimeset}
For sufficiently large $y$ and $x\geq y+1$, any interval $[x-y,x]$ contains a subset of pairwise coprime integers of size at least $ \frac{c_3 \sqrt{y}}{\log y}$ for some positive constant $c_3$.  
\end{theorem}

\begin{proof}
Let $y'=y/c_1$, where $c_1$ is the constant from Lemma \ref{lem:sieve}.  That lemma then implies that the set $S \subset [x-y,x]$ consisting of integers in this interval free of prime factors smaller than $\sqrt{y'}$ contains at least $\frac{y'}{ \log^2 y'}$ integers.  

Now, let $p \geq \sqrt{y'}$ be prime, and suppose $p|n$ for some $n \in S$.  Then $n=pm$ where $m \in \left[\frac{x}{p} - \frac{y}{p},\frac{x}{p}\right]$ (an interval of length $\frac{c_1y'}{p}$).  Since $n$ is free of prime factors smaller than $\sqrt{y'}$, $m$ will be free of such primes as well.  While we can't sieve this shorter interval of primes as large as $\sqrt{y'}$, we can use Lemma \ref{lem:sieve} to sieve this interval of primes up to $\sqrt{\frac{y'}{p}}$, at least so long as $\frac{y'}{p}$ is at least two.  Thus for each prime $\sqrt{y'}\leq p < \frac{y'}{2}$, the number of integers in $S$ divisible by the prime $p$ is at most \[\frac{c_2 y'}{p \log \frac{y'}{p}}.\]
For those primes $\frac{y'}{2}\leq p<y$, we can bound the number of integers in $S$ divisible by $p$ trivially by $\left \lceil \frac{y}{p}\right \rceil =O(1)$.  

We now use Turan's graph theorem to prove that a large subset of $S$ is pairwise coprime.  Construct a graph in which the vertices are the elements of $S$ and the edges connect vertices corresponding to integers which share a prime factor.  Adding together the total number of edges produced by each prime, we find that the total number of edges in the graph is at most \begin{align*} \frac{1}{2}\sum_{\sqrt{y'}\leq p < \frac{y'}{2}} \left(\frac{c_2 y'}{p \log \frac{y'}{ p}}\times \left(\frac{ c_2 y'}{p \log \frac{y'}{ p}}-1\right) \right) &+\frac{1}{2}\sum_{\frac{y'}{2}\leq p <y} \left \lceil \frac{y}{p}\right \rceil\left(\left \lceil \frac{y}{p}\right \rceil-1 \right) \\
&\leq \sum_{\sqrt{y'}\leq p < \frac{y'}{2}} \frac{c_2^2y'^2}{p^2 \log^2 \frac{y'}{p}} + \sum_{\frac{y'}{2} \leq p < y } O(1).\end{align*}
By partial summation this expression is at most  $\frac{c' y'^{3/2}}{\log^3 y'}$ for some constant $c'$.

Turan's graph theorem states that any graph with $v$ vertices and $e$ edges has an independent set of vertices of size at least $\frac{v^2}{v+2e}$.  Applying this to our graph we find there must be an independent set of vertices (corresponding to a set of pairwise coprime integers) of size at least \[\frac{\frac{y'^2}{\log^4 y'}}{\frac{y'}{\log^2 y'} + \frac{2c' y'^{3/2}}{\log^3 y'}} \gg \frac{\sqrt{y'}}{\log y'}\] and the result follows.
\end{proof}

\begin{remark} \label{rem:pfactors}
Note that in the construction above, the integers in the set were free of prime factors less than $\sqrt{y'}$, and thus have at most \[\frac{\log x }{\log \sqrt{y'}} = \frac{2 \log x}{\log y -\log c_1} = \frac{2 \log x}{\log y} + O\left(\frac{\log x}{\log^2 y}\right)\] prime factors.
\end{remark}

\section{Primitive sets without large gaps}

Using these results we are now able to give a proof of Theorem \ref{thm:primitive} using the probabilistic method.

\begin{proof}[Proof of Theorem \ref{thm:primitive}] 
We construct a primitive set according to the following probabilistic construction and then show that, with high probability, the set we constructed does not have any gaps greater than the bound \eqref{primitive bound}.  

Fix $\epsilon>0$. For each prime number $p_i$ we choose a corresponding positive-integer-valued random variable $X_i$ with distribution \[ P(X_i = n) = \frac{C_{\epsilon}}{n\log^{1+\tfrac{\epsilon}{8}} (n+2)},\]
with $C_{\epsilon}$ chosen to normalize the distribution. (Note that the sum of these terms converges since the power on the logarithm is greater than 1. The purpose of adding two inside the logarithm is just to make the probability positive when $n$ is either 1 or 2.)  We then construct the set of integers $Q = \{n \geq 2 : p_i|n \rightarrow \Omega(n) = X_i\}$, consisting of only those integers $n$ for which the total number of prime factors dividing $n$ agrees with the random variable $X_i$ corresponding to every single one of its prime divisors, $p_i$.    

It is readily seen that this construction always produces a primitive set, since if $a \in Q$, and $a|b$ with $b>a$, then $\Omega(a) < \Omega(b)$, but any prime dividing $a$ also divides $b$, and so $b$ cannot be in $Q$.

We now show that we expect every interval of size \eqref{primitive bound} to contain an element of this set.  Let \begin{equation} y = \exp \left(\sqrt{2\log x \log \log  x + (2+\epsilon)\log x \log_3 x }\right), \label{primitivey} \end{equation}
and consider the interval $[x-y,x]$.  
Using Theorem \ref{thm:coprimeset}, along with the observation of Remark \ref{rem:pfactors}, there exists a subset $S$ of the integers in this interval containing at least $\frac{c_3 \sqrt{y}}{\log y}$ integers from this interval which are pairwise coprime.  Furthermore, the integers in $S$ have no more than $\frac{2 \log x}{\log y} + O\left(\frac{\log x}{\log^2 y}\right)$ prime factors.  Suppose $n \in S$, then the probability that $n \in Q$ is \begin{align*}
    P&(n \in Q) = \prod_{p_i|n} P(X_i = \Omega(n)) = \prod_{p_i|n} \frac{C_{\epsilon}}{\Omega(n)\log^{1+\tfrac{\epsilon}{8}} (\Omega(n)+2)} \\
    &\geq \left(\frac{C_{\epsilon}}{\left(\frac{2 \log x}{\log y} + O\left(\frac{\log x}{\log^2 y}\right)\right)\log^{1+\tfrac{\epsilon}{8}}\left(\frac{2 \log x}{\log y} +O(1)\right)}\right)^{\frac{2 \log x}{\log y} + O\left(\frac{\log x}{\log^2 y}\right)} \notag \\
    &= \exp\left(\left(-\frac{2 \log x}{\log y} {+} O\left(\frac{\log x}{\log^{2} y}\right)\right){\times} \left(\log\left(\frac{\log x}{\log y}\right){+}\left(1{+}\tfrac{\epsilon}{8}\right)\log_3 x {+}O_\epsilon\left(1 \right)\right)\right)\notag \\
    &= \exp\left(-\frac{2 \log x}{ \log y}\left(\log\left(\frac{\log x}{\log y}\right) +\left(1+\frac{\epsilon}{8}\right) \log_3 x+ O_\epsilon(1)\right)\right). \notag
\end{align*}

Since the elements of $S$ are pairwise coprime, the probability that any one element of $S$ is included in $Q$ is independent of the probability of any other element is included.  Thus the  probability that no integer from the interval $[x~-~y,x]$ is included in $Q$ can be bounded as follows. \begin{align*}
    &P([x-y,y]\cap Q = \varnothing) \leq P(S\cap Q = \varnothing) = \prod_{n \in S} \left(1-P(n \in Q)\right) \\
    & \leq  \prod_{n \in S}\left(1{-} \exp\left(-\frac{2 \log x}{ \log y}\left(\log\left(\frac{\log x}{\log y}\right) +\left(1+\tfrac{\epsilon}{8}\right) \log_3 x+ O_\epsilon(1)\right)\right)\right) \\
    &\leq \left(1- \exp\left(-\frac{2 \log x}{ \log y}\left(\log\left(\frac{\log x}{\log y}\right) +\left(1{+}\tfrac{\epsilon}{8}\right) \log_3 x+ O_\epsilon(1)\right)\right)\right)^{\frac{c_3 \sqrt{y}}{\log y}} \\
    & \leq \exp\left( - \frac{c_3 \sqrt{y}}{\log y} {\times} \exp\left(-\frac{2 \log x}{ \log y}\left(\log\left(\frac{\log x}{\log y}\right) {+}\left(1{+}\tfrac{\epsilon}{8}\right) \log_3 x{+} O_\epsilon(1)\right)\right)\right) \\
    & = \exp\left(\hspace{-.5mm} {-}\exp \left(\tfrac{1}{2}\log y {-}\log \log y {-} \frac{2\log x}{\log y}\left( \log \left(\frac{\log x}{\log y}\right)\hspace{-.5mm}{+}\left(1{+}\tfrac{\epsilon}{8}\right) \log_3 x {+} O_\epsilon\hspace{-.5mm}(1)\hspace{-.2mm}\right)\hspace{-.2mm}\right)\hspace{-.2mm}\right).
\end{align*}

Now, inserting  our choice \eqref{primitivey} for the length $y$ of the interval, the innermost exponent above becomes
\begin{align*}
   \tfrac{1}{2}&\sqrt{2\log x (\log \log  x {+} \left(1{+}\tfrac{\epsilon}{2}\right) \log_3 x)} {-} \frac{2\log x\left(\log \left(\frac{\sqrt{\log x}}{\sqrt{\log \log  x }}\right) \hspace{-0.5mm}{+}\left(1{+}\frac{\epsilon}{8}\right) \log_3 x{+}O_\epsilon(1)\right)}{\sqrt{2\log x (\log \log  x + \left(1+\frac{\epsilon}{2}\right) \log_3 x)}} \\
   &=\sqrt{\tfrac{1}{2}\log x}\left(\sqrt{\log \log  x {+} \left(1{+}\tfrac{\epsilon}{2}\right) \log_3 x}- \frac{\log \log  x {+} \left(1{+}\tfrac{\epsilon}{2}{-}\frac{\epsilon}{4}\right)\log_3 x+O_\epsilon(1)}{\sqrt{\log \log  x + \left(1{+}\frac{\epsilon}{2}\right) \log_3 x}}\right) \\
   &= \frac{\epsilon}{8}\frac{\sqrt{2\log x}}{\sqrt{\log \log  x}}\left(\log_3 x+O_\epsilon(1)\right).
\end{align*}

Therefore the probability that none of the integers from the interval $[x-y,x]$ are included in $Q$, which is less than the probability that no integer in $S$ is included in $Q$ since $S \subset [x-y,x]$, is at most $\exp\left(-\exp\left(\frac{\epsilon}{8}\frac{\sqrt{2\log x}}{\sqrt{\log \log  x}}\left(\log_3 x+O_\epsilon(1)\right)\right)\right)$.  Using linearity of expectation, and by starting the sequence at a sufficiently high initial value $N$, we can ensure that the expected number of intervals of the form $[x-y,x]$ which do not contain an integer in $Q$ is at most 
\begin{align*}
    \sum_{x>N} P(&[x {-} y,x] \cap Q = \varnothing) \leq \sum_{x>N}  \exp\left(-\exp\left(\frac{\epsilon\sqrt{2\log x}}{8\sqrt{\log \log  x}}\left(\log_3 x{+}O_\epsilon(1)\right)\right)\right) <1
\end{align*} 
since this series converges.

Because the expected number of intervals that do not contain an integer in $Q$ is less than 1, there must exist a sequence $Q$ which intersects every such interval, and thus satisfies the properties of the theorem.
\end{proof}

\section{Geometric-Progression-Free sets without large gaps} \label{sec:gpf}

A very similar construction can be used to prove Theorem \ref{thm:gpf}, producing a set free of 3-term geometric progressions with gaps smaller than those obtained by He.  

\begin{proof}[Proof of Theorem \ref{thm:gpf}] 
Following the method of proof of Theorem \ref{thm:primitive}, we construct a set similar to the squarefree numbers, in the sense that each prime number is only allowed to appear (if it appears at all) to one fixed power in any element of the set. As before, we construct this set probabilistically and then bound the probability that it omits any interval of the size given in \eqref{gp bound}.  

For each prime $p_i$ choose a positive-integer-valued random variable $X_i$ with distribution \[ P(X_i = n) = \frac{1}{2^n}.\]
Now construct the set of integers $T = \{n \geq 2 : p_i|n \rightarrow p_i^{X_i}||n\}$ consisting of those integers $n$ where the exponent on each of its prime divisors $p_i$ is equal to the random variable $X_i$.  (If $p_i$ divides $n$ then $p_i^{X_i}$ is the largest power of $p_i$ that divides $n$.)
 
This set $T$ is free of 3-term geometric progressions of integers for essentially the same reason that the squarefree integers avoid such progressions. If $\{a,ar,ar^2\}$ is any geometric progression with $r\in \mathbb{Q}$, $r>1$ and $p$ divides the numerator of $r$ but not the denominator, then $p$ appears to different, positive, powers in $ar$ and $ar^2$, and hence both cannot be in $T$.

We now show that we expect every interval of size \eqref{gp bound} to contain an element of this set.  Let \begin{equation} y = \exp \left(2\sqrt{\log 2\log x +\tfrac{3+\epsilon}{2}\sqrt{\log 2 \log x}\log \log  x }\right) \label{eq:gpfy} \end{equation}
and consider the interval $[x-y,x]$.  
We again use Theorem \ref{thm:coprimeset} to obtain a pairwise coprime subset $S$ of this interval of size at least $\frac{c_3 \sqrt{y}}{\log y}$ consisting of integers having at most $\frac{2 \log x}{\log y} + O\left(\frac{\log x}{\log^2 y}\right)$ prime factors.  The probability an integer $n$ from this set is contained in $T$ is \begin{align*}
    P(n \in T) = \prod_{p_i^\alpha||n} P(X_i = \alpha) &= \prod_{p_i^\alpha||n} \frac{1}{2^\alpha} = \left(\frac{1}{2}\right)^{\Omega(n)} \geq \left(\frac{1}{2}\right)^{\frac{2 \log x}{\log y} + O\left(\frac{\log x}{\log^2 y}\right)} \\
    &= \exp\left(-\left(\frac{2 \log 2 \log x}{\log y} + O\left(\frac{\log x}{\log^2 y}\right)\right) \right).
\end{align*}

Exploiting the fact that elements of $S$ are pairwise coprime, the probability that none of the elements of $S$ are included in $T$ is \begin{align*}
    \prod_{n \in S} \left(1-P(n \in T)\right) & \leq  \prod_{n \in S}\left(1-\exp\left(-\left(\frac{2 \log 2 \log x}{\log y} + O\left(\frac{\log x}{\log^2 y}\right)\right) \right)\right) \\
    &\leq \left(1-\exp\left(-\left(\frac{2 \log 2 \log x}{\log y} + O\left(\frac{\log x}{\log^2 y}\right)\right) \right)\right)^{\frac{c_3 \sqrt{y}}{\log y}} \\
    & \leq \exp\left( - \frac{c_3 \sqrt{y}}{\log y} \times \exp\left(-\left(\frac{2 \log 2 \log x}{\log y} + O\left(\frac{\log x}{\log^2 y}\right)\right) \right)\right) \\
    & = \exp\left( -\exp \left(\tfrac{1}{2}\log y - \frac{2\log 2 \log x}{\log y}  - \log \log  y + O(1) \right)\right) .
\end{align*}

Inserting \eqref{eq:gpfy} here for $y$ the innermost exponent above becomes
\begin{align*}
   &\sqrt{\log 2\log x +\tfrac{3+\epsilon}2\sqrt{\log 2 \log x}\log \log  x } - \frac{ \log 2 \log x  }{\sqrt{ \log 2 \log x +\frac{3+\epsilon}{2}\sqrt{\log 2 \log x}\log \log  x}} \\
   &\hspace{3cm} - \frac{1}{2} \log \log  x +O(1)\\
   & =\sqrt{\log 2 \log x}\left(\sqrt{1 {+}\frac{(3{+}\epsilon)\log \log  x}{2\sqrt{\log 2\log x}}} - \frac{1}{\sqrt{1 {+}\frac{(3+\epsilon)\log \log  x}{2\sqrt{\log 2\log x}}}}\right) {-} \frac{1}{2} \log \log  x +O(1)\\
   &=\sqrt{\log 2 \log x}\left(\frac{(3+\epsilon)\log \log  x}{2\sqrt{\log 2\log x}}+ O\left(\frac{(\log \log  x)^2}{\log x}\right)\right) - \frac{1}{2} \log \log  x +O(1)\\
   &=\left(1+\frac{\epsilon}{2}\right) \log \log  x   +O\left(1\right).
\end{align*}
The fourth line above was obtained using the Taylor expansion \[\sqrt{1+x}-\frac{1}{\sqrt{1+x}} = x +O(x^2)\] around $x=0$.

Therefore, the probability that no such integer is included in $T$ is at most $\exp\left(-\exp\left(\left(1+\frac{\epsilon}{2}\right) \log \log  x   +O\left(1\right)\right)\right)$.  As before, by linearity of expectation, we may choose a sufficiently high initial value $N$ so that the expected number of intervals of the form $[x-y,x]$ which do not contain an integer in $T$ is at most 
\[\sum_{x>N} P([x - y,x] \cap T = \varnothing) \leq \sum_{x>N} \exp\left(-\exp\left(\left(1+\frac{\epsilon}{2}\right) \log \log  x   +O\left(1\right)\right)\right) <1\] 
since this series is convergent.
Thus there exists a geometric-progression-free sequence $T$ satisfying the properties of the theorem.
\end{proof}

\section{Final Remarks}

While we were able to show that there exist primitive sets in which the gap between consecutive terms was much smaller than what is known to be true, even conditionally for the primes, the method developed doesn't seem to generalize to sequences of pairwise coprime integers.  

\begin{question}
Can one prove that there exists a sequence $v_1<v_2<\cdots$ of pairwise coprime integers in which the difference between consecutive terms $v_n{-}v_{n-1}$ is smaller than the best known upper bound for the gaps between primes?
\end{question}

\section*{aknowledgements}
The author is grateful to Angel Kumchev, Greg Martin and Carl Pomerance for helpful discussions during the development of this paper, and to the anonymous referee for useful feedback.

\bibliographystyle{amsplain}
\bibliography{bibliography}

\end{document}